\documentclass[oneside,english]{amsart}
\usepackage{textcomp}
\usepackage{amsthm}
\usepackage{amssymb}
\usepackage{esint}

\makeatletter
\numberwithin{equation}{section}
\numberwithin{figure}{section}
\theoremstyle{plain}
\newtheorem{thm}{\protect\theoremname}
  \theoremstyle{definition}
  \newtheorem{defn}[thm]{\protect\definitionname}
  \theoremstyle{remark}
  \newtheorem{rem}[thm]{\protect\remarkname}
  \theoremstyle{plain}
  \newtheorem{lem}[thm]{\protect\lemmaname}
  \theoremstyle{plain}
  \newtheorem{prop}[thm]{\protect\propositionname}
  \theoremstyle{plain}
  \newtheorem{cor}[thm]{\protect\corollaryname}

\makeatother

\usepackage{babel}
  \providecommand{\corollaryname}{Corollary}
  \providecommand{\definitionname}{Definition}
  \providecommand{\lemmaname}{Lemma}
  \providecommand{\propositionname}{Proposition}
  \providecommand{\remarkname}{Remark}
\providecommand{\theoremname}{Theorem}

\begin{document}

\title[Waves in three species system]{traveling waves in a three species competition-cooperation system }

\author{xiaojie hou$^{1}$, yi li$^{2}$}

\address{$^{1}$Department of Mathematics and Statistics, Unievrsity of North
Carolina Wilmington, Wilmington, NC 28411.$^{2}$Department of Mathematics
\& Statistics, Wright State University , Dayton, OH, Dayton, OH 45435.}

\email{houx@uncw.edu, yi.li@wright.edu}

\begin{abstract}
This paper studies the traveling wave solutions to a three species
competition cooperation system. The existence of the traveling waves
is investigated via monotone iteration method. The upper and lower
solutions come from either the waves of KPP equation or those of certain
Lotka Volterra system. We also derive the asymptotics and uniqueness
of the wave solutions. The results are then applied to a Lotka Volterra
system with spatially averaged and temporally delayed competition.
\end{abstract}
\maketitle

\section{introduction}

We study the traveling wave solutions of the following three species
competition cooperation system

\begin{equation}
\left\{ \begin{array}{l}
u_{t}=u_{xx}+u(1-u-a_{1}w),\\
u_{t}=v_{xx}+rv(1-a_{2}u-v),\\
w_{t}=w_{xx}+{\displaystyle \frac{1}{\tau}}(v-w),
\end{array}\quad\quad(x,t)\in\mathbb{R}\times\mathbb{R}^{+}.\right.\label{eq:1.1}
\end{equation}
where $u(x,t)$, $v(x,t)$ and $w(x,t)$ stand for the population
densities of the three different species, $a_{i}>0$ is interaction
constant, $i=1,2$ and $r>0$ ($-\frac{1}{\tau}<0$, respectively)
is the relative intrinsic growth rate of the species $v$ ($w$, respectively).
Aside from the intra-specific competitions, system (\ref{eq:1.1})
describes the relation that the species $w$ competes with $u$ and
$u$ competes with $v$, while $v$ cooperates with $w$.

The purpose of our study is of two folds: there are less results (see\cite{Miller,Hung})
on the traveling wave solutions to the three species systems even
for one with simple form as (\ref{eq:1.1}); on the other hand we
would like to extend the results of \cite{Ashwin} from the tempo-spatial
delayed KPP (Kolmogorov, Petrovsky and Piscounov) equation to the
following Lotka Volterra competition system with spatial temporal
delay

\begin{equation}
\left\{ \begin{array}{lll}
u_{t} & = & u_{xx}+u(1-u-a_{1}g**v),\\
\\
v_{t} & = & v_{xx}+rv(1-a_{2}u-v),
\end{array}\right.\label{eq:1.2}
\end{equation}
where the function $g**v=\int_{-\infty}^{+\infty}\int_{-\infty}^{t}\frac{1}{\sqrt{2\pi(t-s)}}e^{-\frac{(x-y)^{2}}{4(t-s)}}\frac{1}{\tau}e^{-\frac{(t-s)}{\tau}}v(s)dsdy$
represents the tempo-spatial delay of the response of species $v$
to $u$, here $g(x,t)=\frac{1}{\sqrt{4\pi t}}e^{\frac{x^{2}}{4t}}\frac{1}{\tau}e^{-\frac{t}{\tau}}$,
and satisfies

\begin{equation}
g**1=1.\label{eq:1.3}
\end{equation}

Noting that if $g**v$ is replaced by $v$ in (\ref{eq:1.2}), we
recover the classical Lotka Volterra competition system, and fruitful
results have been devoted in the study of traveling waves \cite{BoumenirNguyen,FeiCarr,Huang,Hosono,Kan-on,Kan-on-1,Kanel,KanelZhou,LeungHouFeng,Volpert,WuZou}
arising from it. It is interesting to see the long term effect of
introducing spatio-temporal delay to the competition. The temporal
delay accounts for the time once consumed resource the dominated species
needs to wait for its re-growth, and the spatial averaging accounts
for the fact that individuals are moving around and have therefore
not been at the same point in space at different times in their history
(see \cite{Ashwin}). 

On setting
\begin{equation}
w(x,t)=\frac{1}{\tau}\int_{-\infty}^{+\infty}\int_{-\infty}^{t}\frac{1}{\sqrt{2\pi(t-s)}}e^{\frac{(x-y)^{2}}{4(t-s)}}e^{-\frac{(t-s)}{\tau}}v(s)dsdy,\label{eq:1.4}
\end{equation}
we easily verify that $w(x,t)$ satisfies the following equation
\begin{equation}
w_{t}=w_{xx}+\frac{1}{\tau}(v-w).\label{eq:1.5}
\end{equation}

System (\ref{eq:1.2}) is now recasted into (\ref{eq:1.1}). The existence
of the traveling wave solution of (\ref{eq:1.1}) is equivalent to
that of (\ref{eq:1.2}).

Throughout the paper we make the following assumptions:

\begin{equation}
\mbox{Condition\,\ H1}:\,0<a_{1}<1<a_{2}\label{Condition H1}
\end{equation}
and either

\begin{equation}
\mbox{Condition \,\ H2a}:\: r(a_{2}-1)<(1-a_{1})\label{ConditionH2a}
\end{equation}
or 

\begin{equation}
\mbox{Condition \,\ H2b}:\: r(a_{2}-1)\geq(1-a_{1})\geq r(a_{1}a_{2}-1)\label{CoditionH2b}
\end{equation}

\begin{equation}
\mbox{Condition \,\ H3}:\: a_{2}>0\mbox{ is suitably large }.\label{eq:ConditionH3}
\end{equation}

Under the condition (\ref{Condition H1}), system (\ref{eq:1.1})
admits three constant equilibria: $(0,0,0)$, $(0,1,1)$ and $(1,0,0)$,
and the first two are unstable while the third is stable. We are interested
in finding the traveling wave solutions of (\ref{eq:1.1}) connecting
$(0,1,1)$ with $(1,0,0)$. Via transformation (\ref{eq:1.4}), the
existence of the traveling was solutions of (\ref{eq:1.1}) connecting
$(0,1,1)$ to $(1,0,0)$ is equivalent to that of (\ref{eq:1.2})
connecting $(0,1)$ with $(1,0)$.

A traveling wave solution for (\ref{eq:1.1}) has the form $(u(x,t),v(x,t),w(x,t))=(u(x+ct),v(x+ct),w(x+ct))=(u(\xi),v(\xi),w(\xi))$,
$\xi=x+ct$, and satisfies the system

\begin{equation}
\left\{ \begin{array}{l}
u_{\xi\xi}-cu_{\xi}+u(1-u-a_{1}w)=0,\\
v_{\xi\xi}-cv_{\xi}+rv(1-a_{2}u-v)=0,\\
w_{\xi\xi}-cw_{\xi}+{\displaystyle \frac{1}{\tau}}(v-w)=0,\\
(u,v,w)(-\infty)=(0,1,1),\quad(w,v,w)(+\infty)=(1,0,0).
\end{array}\right.\label{eq:1.10}
\end{equation}

For the convenience of later study, we change (\ref{eq:1.10}) into
monotone (\cite{Volpert}). Let $\bar{u}=u$, $\bar{v}=1-v$ and $\bar{w}=1-w$,
and drop the bars on $u$, $v$ and $w$, we arive at

\begin{equation}
\left\{ \begin{array}{l}
u_{\xi\xi}-cu_{\xi}+u(1-a_{1}-u+a_{1}w)=0,\\
v_{\xi\xi}-cv_{\xi}+r(1-v)(a_{2}u-v)=0,\\
w_{\xi\xi}-cw_{\xi}+{\displaystyle \frac{1}{\tau}}(v-w)=0,\\
(u,v,w)(-\infty)=(0,0,0),\quad(u,v,w)(+\infty)=(1,1,1).
\end{array}\right.\label{eq:1.11}
\end{equation}

Conditions (\ref{ConditionH2a}) and (\ref{CoditionH2b}) stipulate
that wave solution is either below or above the plane $u=v$, and
this is reflected in the construction of the upper and lower solutions
in section \ref{sec:3}. 
\begin{defn}
\label{Def1}A $C^{2}(\mathbb{R})\times C^{2}(\mathbb{R})$ function
$(\bar{u}(\xi),\bar{v}(\xi),\bar{w}(\xi))$, $\xi\in\mathbb{R}$ is
an upper solution of (\ref{eq:1.11}) if it satisfies the inequalities
\begin{equation}
\left\{ \begin{array}{l}
u_{\xi\xi}-cu_{\xi}+u(1-a_{1}-u+a_{1}w)\leq0,\\
v_{\xi\xi}-cv_{\xi}+r(1-v)(a_{2}u-v)\leq0,\\
w_{\xi\xi}-cw_{\xi}+{\displaystyle \frac{1}{\tau}}(v-w)\leq0,
\end{array}\right.\label{eq:1.12}
\end{equation}
and the boundary conditions 
\begin{equation}
\left(\begin{array}{c}
u\\
v\\
w
\end{array}\right)(-\infty)\geq\left(\begin{array}{c}
0\\
0\\
0
\end{array}\right),\,\;\left(\begin{array}{c}
u\\
v\\
w
\end{array}\right)(+\infty)\geq\left(\begin{array}{c}
1\\
1\\
1
\end{array}\right).\label{eq:1.13}
\end{equation}

A lower solution of (\ref{eq:1.11}) is defined in a similar way by
reversing the inequalities in (\ref{eq:1.12}) and (\ref{eq:1.13}). \end{defn}
\begin{rem}
\label{Remark2}The minimum of two upper solutions is still an upper
solution, i.e. suppose $W_{1}(\xi)=(w_{1},w_{2},w_{3})(\xi)$ and
$\bar{W}(\xi)=(\bar{w}_{1},\bar{w}_{2},\bar{w}_{3})(\xi)$, $\xi\in\mathbb{R}$
are two upper solutions, then $V(\xi)=(\min(w_{1}(\xi),\bar{w}_{1}(\xi)),\min(w_{2}(\xi),\bar{w}_{2}(\xi)),\min(w_{3}(\xi),\bar{w}_{3}(\xi)))$
is also an upper solution. For the lower solutions, we have similar
observation except that the minimum is replaced by maximum. 
\end{rem}
Since (\ref{eq:1.11}) is a monotone system, the monotone iteration
method (\cite{WuZou}) is ready to apply once the orderness of the
upper and lower solutions is estabished. The key to the monotone iteration
is to identify a pair of ordered upper and lower solutions to (\ref{eq:1.11})
(\cite{WuZou,BoumenirNguyen}). There are two methods to set up the
upper and lower solutions. The first one was used in \cite{BoumenirNguyen,Huang,WuZou},
which consists of a pair non-smooth upper and lower solutions, and
the similar idea was later successfully generalized to handle local
and nonlocal equations, the second one is based on a pair of smooth
upper and lower solutions from known equations, and the method was
applied in \cite{LeungHouFeng} for a general form of two species
Lotka Volterra competition system and in \cite{HouFeng} for a model
system arising from game theory. We will use the ideas of the second
method to set up the upper and lower solutions for (\ref{eq:1.11}).
See section 2 for details.

The other interesting aspects of the traveling front solutions are
the minimal wave speed, the uniqueness, the asymptotics and the stability.
The minimal wave speed is also referred to as the critical wave speed,
below which the there will be no monotonic traveling waves. Also the
traveling waves with the critical speed behaves differently at $-\infty$,
see \cite{LeungHouFeng}. We will use a generalized version of sliding
domain method (see \cite{Berestycki} ) to show the uniqueness of
the front solution corresponding to each speed. 

This paper is organized as follows: In section \ref{sec:2} we gather
the necessary information about the KPP and Lotka Volterra waves,
in particular we derive the asymptotics of the Lotka Volterra waves
at $-\infty$ which is the key in setting up the upper and lower solutions
for (\ref{eq:1.11}); In section \ref{sec:3} we show the existence
of the wave solutions for (\ref{eq:1.11}) and further derive their
properties such as strict monotonicity, the uniqueness and the asymptotics.

\section{\label{sec:2}properties of waves for kpp and a two species lotka
volterra competition equations}

In this section, we introduce porperties of the wave solutions to
KPP equation and to a two species Lotka Volterra competition system,
which will be a key ingredient in the construction of the upper and
lower solutions for system (\ref{eq:1.11}). For the rest of the paper
the inequality between two vectors is component-wise. 

The construction of the smooth upper and lower solution pairs for
system (\ref{eq:1.11}) is based on the known results on the KPP equations
and the recent results on a rescaled Lotka Volterra system. It seems
that the asymptotics of the Lotka Volterra waves derived in this section
is new.

Consider the following form of the KPP equation: 
\begin{equation}
\left\{ \begin{array}{l}
\omega''-c\omega'+f(\omega)=0,\\
\\
\omega(-\infty)=0,\quad\omega(+\infty)=b,
\end{array}\right.\label{eq:2.1}
\end{equation}
where $f\in C^{2}([0,\, b])$ and $f>0$ on the open interval $(0,b)$
with $f(0)=f(b)=0$, $f'(0)=\bar{a}>0$ and $f'(b)=-b_{1}<0$. We
first recall the following result (\cite{Sattinger}):
\begin{lem}
\label{lem:3}Corresponding to every $c\geq2\sqrt{\bar{a}}$, system
(\ref{eq:2.1}) has a unique (up to a translation of the origin) monotonically
increasing traveling wave solution $\omega(\xi)$ for $\xi\in\mathbb{R}$.
The traveling wave solution $\omega(\xi)$, $\xi\in\mathbb{R}$ has
the following asymptotic behaviors:

For the wave solution with non-critical speed $c>2\sqrt{\bar{a}}$,
we have 
\begin{equation}
\omega(\xi)=a_{\omega}e^{\frac{c-\sqrt{c^{2}-4\bar{a}}}{2}\xi}+o(e^{\frac{c-\sqrt{c^{2}-4\bar{a}}}{2}\xi})\mbox{ as }\xi\rightarrow-\infty,\label{eq:2.2}
\end{equation}
 
\begin{equation}
\omega(\xi)=b-b_{\omega}e^{\frac{c-\sqrt{c^{2}+4b_{1}}}{2}\xi}+o(e^{\frac{c-\sqrt{c^{2}+4b_{1}}}{2}\xi})\mbox{ as }\xi\rightarrow+\infty,\label{eq:2.3}
\end{equation}
 where $a_{\omega}$ and $b_{\omega}$ are positive constants.

For the wave with critical speed $c=2\sqrt{\bar{a}}$, we have

\begin{equation}
\omega(\xi)=(a_{c}+d_{c}\xi)e^{\sqrt{\bar{a}}\xi}+o(\xi e^{\sqrt{\bar{a}}\xi})\mbox{ as }\xi\rightarrow-\infty,\label{eq:2.4}
\end{equation}

\begin{equation}
\omega(\xi)=b-b_{c}e^{(\sqrt{\bar{a}}-\sqrt{\bar{a}+b_{1}})\xi}+o(e^{(\sqrt{\bar{a}}-\sqrt{\bar{a}+b_{1}})\xi})\mbox{ as }\xi\rightarrow+\infty,\label{eq:2.5}
\end{equation}
 where the constant $d_{c}$ is negative, $b_{c}$ is positive, and
$a_{c}\in\mathbb{R}$. 
\end{lem}
We also need the existence and asymptotics (at $-\infty$) of the
solutions of the following rescaled version of Lotka Volterra system:

\begin{equation}
\left\{ \begin{array}{l}
u_{\xi\xi}-cu_{\xi}+u(1-a_{1}-u+a_{1}w)=0,\\
\\
v_{\xi\xi}-cv_{\xi}+r(1-v)(a_{2}u-v)=0,\\
\\
(u,v)(-\infty)=(0,0),\quad(u,v)(+\infty)=(1,1).
\end{array}\right.\label{eq:2.6}
\end{equation}

The asymptotics of the wave solutions will be derived by comparing
the asymptotic decay rates of the upper and lower solutions of (\ref{eq:2.6}).
Recalling that $(\bar{u},\bar{v})(\xi)\in C^{2}(\mathbb{R})\times C^{2}(\mathbb{R})$,
$\xi\in\mathbb{R}$ is an upper solution for (\ref{eq:2.6}) if it
satisfies 

\begin{equation}
\left\{ \begin{array}{l}
u_{\xi\xi}-cu_{\xi}+u(1-a_{1}-u+a_{1}w)\leq0,\\
\\
v_{\xi\xi}-cv_{\xi}+r(1-v)(a_{2}u-v)\leq0,\\
\\
(u,v)(-\infty)\geq(0,0),\quad(u,v)(+\infty)\geq(1,1).
\end{array}\right.\label{eq:2.7}
\end{equation}
and we can define the lower solution for (\ref{eq:2.6}) by reversing
the above inequalities. 

We have the following conclusions:
\begin{lem}
\label{lem:4}Let the parameters satisfy either H1, H2a or H1, H2b,
then for each $c\geq2\sqrt{1-a_{1}}$ , system (\ref{eq:2.6}) has
a unique (up to a shift of the origin) strictly monotonic solution,
and for $0\leq c<2\sqrt{1-a_{1}}$, (\ref{eq:2.6}) does not have
monotonic solution. At $-\infty$ the solution has the following asymptotical
properties:

For $c>2\sqrt{1-a_{1}}$, the solution $(u(\xi),v(\xi))$ satisfies,
as $\xi\rightarrow-\infty$; 
\begin{equation}
\left(\begin{array}{c}
u(\xi)\\
\\
v(\xi)
\end{array}\right)=\left(\begin{array}{c}
A_{1}\\
\\
A_{2}
\end{array}\right)e^{\frac{c-\sqrt{c^{2}-4(1-a_{1})}}{2}\xi}+o(e^{\frac{c-\sqrt{c^{2}-4(1-a_{1})}}{2}\xi}).\label{eq:2.8}
\end{equation}

while for $c_{\mbox{ }}^{*}=2\sqrt{1-a_{1}}$, the solution $(u(\xi),v(\xi))^{T}$
satisfies 
\begin{equation}
\left(\begin{array}{c}
u(\xi)\\
\\
v(\xi)
\end{array}\right)=\left(\begin{array}{c}
A_{11c}+A_{12c}\xi\\
\\
A_{21c}+A_{22c}\xi
\end{array}\right)e^{\sqrt{1-a_{1}}\xi}+o(\xi e^{\sqrt{1-a_{1}}\xi})\label{eq:2.9}
\end{equation}
as $\xi\rightarrow-\infty$, where $A_{1},A_{2}>0$, $A_{11c},A_{21c}\in\mathbb{R}$
and $A_{12c},A_{22c}<0$.\end{lem}
\begin{proof}
The existence of the waves under conditions H1-H2a, or H1, H2b is
contained in \cite{LiWeinbergerLewis}, and also refer to \cite{LeungHouFeng}
for the existence and the asymptotics of the waves in a more general
form of a two species competition system under conditions H1-H2a.

From now on we will concentrate on system (\ref{eq:2.6}) under conditions
H1 and H2b. The authors in \cite{LiWeinbergerLewis} proved the existence
of the monotone solutions for $c\geq2\sqrt{1-a_{1}}$ by showing that
(\ref{eq:2.6}) is linearly determinated. However such method does
not bring us the crucial information on the aymptotics of the wave
solutions that is needed in section \ref{sec:3}. We show that the
traveling wave as derived in \cite{LiWeinbergerLewis} is actually
squeezed by the lower and upper solutions of (\ref{eq:2.6}) constructed
below. Noting the upper and lower solutions differ from that in \cite{Huang}. 

We first set up the lower solution for (\ref{eq:2.6}).

For a fixed $c\geq2\sqrt{1-a_{1}}$, let $\underline{u}(\xi)$, $\xi\in\mathbb{R}$
be a corresponding solution of the following KPP equation 

\begin{equation}
\left\{ \begin{array}{l}
u''-cu'+(1-a_{1})u(1-{\displaystyle u)=0,}\\
\\
u(-\infty)=0,\quad u(+\infty)={\displaystyle 1.}
\end{array}\right.\label{eq:2.10}
\end{equation}

It is straightforward to verify the following claim.

\noindent \textit{Claim A. Under conditions H1 and H2b, for each fixed
$c\geq2\sqrt{1-a_{1}}$, the smooth function $(\underline{u},\underline{v})(\xi)\doteq(\underline{u},\underline{u})(\xi)$,
$\xi\in\mathbb{R}$ defines a lower solution for (\ref{eq:2.6}).}

We next set up an super solution for (\ref{eq:2.6}). Choosing a small
number $l$ such that
\begin{equation}
0<l\leq{\displaystyle \frac{1-a_{1}-r(a_{1}a_{2}-1)}{1+r-a_{1}}}.\label{eq:2.11}
\end{equation}
Let $c\geq2\sqrt{1-a_{1}}$ be fixed and $\hat{u}(\xi)$ be a solution
of the following modified KPP equation:
\begin{equation}
\left\{ \begin{array}{l}
u''-cu'+(1-a_{1})u(1-{\displaystyle \frac{l}{1-a_{1}}u)=0,}\\
\\
u(-\infty)=0,\quad u(+\infty)={\displaystyle \frac{1-a_{1}}{l}.}
\end{array}\right.\label{eq:2.12}
\end{equation}
Condition H2b (see \ref{CoditionH2b}) implies that $\frac{1-a_{1}-r(a_{1}a_{2}-1)}{1+r-a_{1}}<1-a_{1}$
we therefore have $\frac{1-a_{1}}{l}>1$.

Setting 
\begin{equation}
(\bar{u},\bar{v})(\xi)=\left\{ \begin{array}{l}
(\hat{u}(\xi),{\displaystyle \frac{1-l}{a_{1}}\hat{u}(\xi)),\quad\mbox{if}\,\hat{u}(\xi)\leq\frac{a_{1}}{1-l};}\\
\\
(\hat{u}(\xi),1),\quad\mbox{if}\,{\displaystyle \frac{a_{1}}{1-l}\leq\hat{u}(\xi)\leq1;}\\
\\
(1,1),\quad\mbox{if}\,\hat{u}(\xi)\geq1.
\end{array}\right.\label{eq:2.13}
\end{equation}
\textit{Claim B . Assume conditions in Claim A and let $l$ satisfy
(\ref{eq:2.11}), then (\ref{eq:2.13}) defines an upper solution
for (\ref{eq:2.6}).}

\noindent \begin{flushleft}
\textit{Proof of claim B.} It is easy to see that $(\bar{u},\bar{v})(-\infty)=(0,0)$,
$(\bar{u},\bar{v})(+\infty)=(1,1)$, and $(u,v)(\xi)=(1,1)$, $\xi\in\mathbb{R}$
solves the first two equations of (\ref{eq:2.6}). We next verify
that $(u,v)(\xi)=(\hat{u}(\xi),{\displaystyle \frac{1-l}{a_{1}}\hat{u}(\xi))}$
is also an upper solution for (\ref{eq:2.6}).
\par\end{flushleft}

For the $u$ component we have
\[
\begin{array}{ll}
 & \hat{u}''-c\hat{u}'+\hat{u}(1-a_{1}-\hat{u}+a_{1}\hat{v})\\
\\
= & \hat{u}''-c\hat{u}'+\hat{u}(1-a_{1}-\hat{u}+(1-l)\hat{u})\\
\\
= & \hat{u}''-c\hat{u}'+(1-a_{1})\hat{u}(1-\frac{l}{1-a_{1}}\hat{u})\\
\\
 & -(1-a_{1})\hat{u}(1-\frac{l}{1-a_{1}}\hat{u})+\hat{u}(1-a_{1}-\hat{u}+a_{1}\hat{v})\\
\\
= & \hat{u}[1-a_{1}-\hat{u}+(1-l)\hat{u}-(1-a_{1})(1-\frac{l}{1-a_{1}}\hat{u})]=0,
\end{array}
\]
 and for the $v$ component we have
\begin{equation}
\begin{array}{ll}
 & \hat{v}_{\xi\xi}-c\hat{v}_{\xi}+r(1-\hat{v})(a_{2}\hat{u}-\hat{v})\\
\\
= & \frac{1-l}{a_{1}}\hat{u}''-c\frac{1-l}{a_{1}}\hat{u}'+r(1-\frac{1-l}{a_{1}}\hat{u})(a_{2}\hat{u}-\frac{1-l}{a_{1}}\hat{u})\\
\\
= & \frac{1-l}{a_{1}}[\hat{u}''-c\hat{u}'+\frac{a_{1}}{1-l}r(1-\frac{1-l}{a_{1}}\hat{u})(a_{2}-\frac{1-l}{a_{1}})\hat{u}\\
\\
 & +(1-a_{1})\hat{u}(1-\frac{l}{1-a_{1}}\hat{u})-(1-a_{1})\hat{u}(1-\frac{l}{1-a_{1}}\hat{u})]\\
\\
= & \frac{1-l}{a_{1}}[\frac{a_{1}}{1-l}r(1-\frac{1-l}{a_{1}}\hat{u})(a_{2}-\frac{1-l}{a_{1}})\hat{u}-(1-a_{1})\hat{u}(1-\frac{l}{1-a_{1}}\hat{u})]\\
\\
= & \frac{1-l}{a_{1}}\hat{u}[r(\frac{a_{2}a_{1}}{1-l}-1)-(1-a_{1})+\hat{u}(l-r(a_{2}-\frac{1-l}{a_{1}}))].
\end{array}\label{eq:2.14}
\end{equation}
To ensure the non-positiveness of the last expression in (\ref{eq:2.14}),
we require: 
\begin{equation}
r(\frac{a_{2}a_{1}}{1-l}-1)-(1-a_{1})\leq0,\label{eq:2.15}
\end{equation}
and either 
\begin{equation}
l-r(a_{2}-\frac{1-l}{a_{1}})\leq0,\label{eq:2.16}
\end{equation}
or
\begin{equation}
\begin{array}{c}
l-r(a_{2}-\frac{1-l}{a_{1}})\geq0,\ \mbox{and}\\
r(\frac{a_{2}a_{1}}{1-l}-1)-(1-a_{1})+\hat{u}(\xi)(l-r(a_{2}-\frac{1-l}{a_{1}}))\leq0\,\mbox{ for all }\xi\in\mathbb{R}.
\end{array}\label{eq:2.17}
\end{equation}
It is a tedious but straightforward verification that for $l$ satisfy
(\ref{eq:2.11}), we have (\ref{eq:2.15}) and (\ref{eq:2.16}), or
(\ref{eq:2.15}) and (\ref{eq:2.17}). Hence in either case the last
expression in (\ref{eq:2.14}) is less than or equal to $0$. 

Analogous to Remark \ref{Remark2} we see that (\ref{eq:2.13}) is
an upper solution. This finishes the proof of the Claim B.

Noting the relation 
\begin{equation}
\hat{u}(\xi)={\displaystyle \frac{1-a_{1}}{l}\underline{u}(\xi)>\underline{u}(\xi)}\quad\quad\xi\in\mathbb{R},\label{eq:2.18}
\end{equation}
there follows the orderness of the upper and lower solution 
\begin{equation}
(\underline{u}(\xi),\underline{v}(\xi))\leq(\bar{u}(\xi),\bar{v}(\xi))\quad\quad\xi\in\mathbb{R}.\label{eq:2.19}
\end{equation}
By the monotone iteration scheme (\cite{WuZou}), for each fixed $c\geq2\sqrt{1-a_{1}}$
the traveling wave solution $(u,v)(\xi)$ satisfies 
\begin{equation}
(\underline{u},\underline{v})(\xi)\leq(u,v)(\xi)\leq(\bar{u},\bar{v})(\xi),\qquad\xi\in\mathbb{R}.\label{eq:2.20}
\end{equation}
Let $c\geq2\sqrt{1-a_{1}}$ be fixed and let $(u(\xi),v(\xi))$, $\xi\in\mathbb{R}$
be a corresponding solution. The upper and lower solutions have the
exactly asymptotic decay rate at $-\infty$by Lemma \ref{lem:3},
the estimates (\ref{eq:2.8}) and (\ref{eq:2.9}) then readily follow.
\end{proof}

\section{\label{sec:3}existence of the waves}

\subsection{\label{sub:3.1}Ordered upper and lower solutions under conditions
H1 and H2a.}

To construct the upper-solution for the system (\ref{eq:1.11}) in
this case, we begin with the following form of KPP system: 
\begin{equation}
\left\{ \begin{array}{l}
\tilde{u}''-c\tilde{u}'+{\displaystyle (1-a_{1})\tilde{u}(1-\tilde{u})}=0,\\
\\
\tilde{v}(-\infty)=0,\quad\tilde{v}(+\infty)=1,
\end{array}\right.\label{eq:3.1}
\end{equation}
where relating to (\ref{eq:2.1}), $f(\tilde{v})=(1-a_{1})\tilde{u}(1-\tilde{u})>0$
for $\tilde{u}\in(0,1)$. $f(0)=f(1)=0$, $f'(0)=(1-a_{1})>0$ and
$f'(1)=-(1-a_{1})<0$. According to Lemma \ref{lem:3}, for each fixed
$c\geq2\sqrt{1-a_{1}}$, system (\ref{eq:3.1}) has a unique (up to
a translation of the origin) traveling wave solution $\bar{u}(\xi)$
satisfying the given boundary conditions. Define 
\begin{equation}
\left(\begin{array}{c}
\bar{u}(\xi)\\
\\
\bar{v}(\xi)\\
\\
\bar{w}(\xi)
\end{array}\right)=\left(\begin{array}{c}
\bar{u}(\xi)\\
\\
\bar{u}(\xi)\\
\\
\bar{u}(\xi)
\end{array}\right),\quad\xi\in\mathbb{R},\label{eq:3.2}
\end{equation}
 we have the following result, 
\begin{lem}
\label{lem:5}Assume the conditions H1 and H2a, for each fixed $c\geq2\sqrt{1-a_{1}}$,
(\ref{eq:3.2}) is a smooth upper solution for system (\ref{eq:1.11}). \end{lem}
\begin{proof}
On the boundary, one has $(\bar{u},\bar{v},\bar{w})(-\infty)=(0,0,0)$,
$(\bar{u},\bar{v},\bar{w})^{T}(+\infty)=(1,1,1)$.

As for the $u$ component, we have

\[
\begin{array}{ll}
 & \bar{u}''-c\bar{u}'+\bar{u}(1-a_{1}-\bar{u}+a_{1}\bar{w})\\
\\
= & \bar{u}''-c\bar{u}'+\bar{u}(1-a_{1}-\bar{u}+a_{1}\bar{u})\\
\\
= & \bar{u}''-c\bar{u}'+(1-a_{1})\bar{u}(1-\bar{u})\\
\\
= & 0,
\end{array}
\]
and for the $v$ component 
\[
\begin{array}{ll}
 & \bar{v}''-c\bar{v}'+r(1-\bar{v})(a_{2}\bar{u}-\bar{v})\\
\\
= & \bar{u}''-c\bar{u}'+r(1-\bar{u})(a_{2}\bar{u}-\bar{u})\\
\\
+ & (1-a_{1})\bar{u}(1-\bar{u})-(1-a_{1})\bar{u}(1-\bar{u})\\
\\
= & [r(a_{2}-1)-(1-a_{1})]\bar{u}(1-\bar{u})\leq0\\
\\
\end{array}
\]
 due to the condition H2a.

As for the $\bar{w}$ component, 

\[
\begin{array}{ll}
 & \bar{w}''-c\bar{w}'+\frac{1}{\tau}(\bar{v}-\bar{w})\\
\\
= & \bar{u}''-c\bar{u}'+(1-a_{1})\bar{u}(1-\bar{u})-(1-a_{1})\bar{u}(1-\bar{u})\\
\\
= & -(1-a_{1})\bar{u}(1-\bar{u})\leq0.
\end{array}
\]

Thus the conclusion follows.
\end{proof}
We next construct the lower solution pair for system (\ref{eq:1.11}).
For any small but fixed number $l$ with

\begin{equation}
0\leq l<{\displaystyle \frac{ra_{2}}{1-a_{1}+r}},\label{eq:3.3}
\end{equation}
 we choose a number $\bar{l}$ such that 

\begin{equation}
0\leq\bar{l}\leq\frac{l}{(1-a_{1})\tau+1}.\label{eq:3.4}
\end{equation}

We begin with yet another KPP system: 
\begin{equation}
\left\{ \begin{array}{l}
\check{u}''-c\check{u}'+{\displaystyle (1-a_{1})\check{u}(1-\frac{1+l}{1-a_{1}}\check{u})}=0,\\
\\
\check{u}(-\infty)=0,\quad\check{u}(+\infty)={\displaystyle \frac{1-a_{1}}{1+l}}<1.
\end{array}\right.\label{eq:3.5}
\end{equation}
 Corresponding to the notions in Lemma \ref{lem:3}, 
\[
f(\check{u})=(1-a_{1})\check{u}(1-\frac{1+l}{1-a_{1}}\check{u})>0
\]
 for $\check{u}\in(0,{\displaystyle \frac{1-a_{1}}{1+l}})$. $f(0)=f({\displaystyle \frac{1-a_{1}}{1+l}})=0$,
$f'(0)=1-a_{1}>0$, and $f'({\displaystyle \frac{1-a_{1}}{1+l}})<0$.
For each $c\geq2\sqrt{1-a_{1}}$ let $\underline{u}(\xi)$, $\xi\in\mathbb{R}$
be a solution of (\ref{eq:3.5}) and let 
\begin{equation}
\left(\begin{array}{c}
\underline{u}(\xi)\\
\\
\underline{v}(\xi)\\
\\
\underline{w}(\xi)
\end{array}\right)=\left(\begin{array}{c}
\underline{u}(\xi)\\
\\
l\underline{u}(\xi)\\
\\
\bar{l}\underline{u}(\xi)
\end{array}\right),\quad\xi\in\mathbb{R},\label{eq:3.6}
\end{equation}
 we have 
\begin{lem}
\label{lem:6}For each $c\geq2\sqrt{1-a_{1}}$ , (\ref{eq:3.6}) is
a smooth lower solution of system (\ref{eq:1.11}). \end{lem}
\begin{proof}
On the boundary, one has 
\[
\left(\begin{array}{c}
\underline{u}(-\infty)\\
\\
\underline{v}(-\infty)\\
\\
\underline{w}(-\infty)
\end{array}\right)=\left(\begin{array}{c}
0\\
\\
0\\
\\
0
\end{array}\right),\quad\left(\begin{array}{c}
\underline{u}(+\infty)\\
\\
\underline{v}(+\infty)\\
\\
\underline{w}(+\infty)
\end{array}\right)=\left(\begin{array}{c}
{\displaystyle \frac{1-a_{1}}{1+l}}\\
\\
l{\displaystyle \frac{1-a_{1}}{1+l}}\\
\\
\bar{l}{\displaystyle \frac{1-a_{1}}{1+l}}
\end{array}\right)<\left(\begin{array}{c}
1\\
\\
1\\
\\
1
\end{array}\right).
\]

Furthermore, for the $u$ component, 
\[
\begin{array}{ll}
 & {\displaystyle \underline{u}''-c\underline{u}'+\underline{u}(1-a_{1}-\underline{u}+a_{1}\underline{w})}\\
\\
= & {\displaystyle \underline{u}''-c\underline{u}'+\underline{u}(1-a_{1}-\underline{u}+a_{1}\bar{l}\underline{u})}\\
\\
 & +(1-a_{1})\underline{u}(1-{\displaystyle \frac{1+l}{1-a_{1}}}\underline{u})-(1-a_{1})\underline{u}(1-{\displaystyle \frac{1+l}{1-a_{1}}}\underline{u})\\
\\
= & \underline{u}[1-a_{1}-\underline{u}+a_{1}\bar{l}\underline{u}-(1-a_{1})+(1+l)\underline{u}]\\
\\
= & \underline{u}^{2}(a_{1}\bar{l}+l)\\
\\
> & 0,
\end{array}
\]
and for the $v$ component, we have

\[
\begin{array}{ll}
 & \underline{v}''-c\underline{v}'+r(1-\underline{v})(a_{2}\underline{u}-\underline{v})\\
\\
= & l[\underline{u}''-c\underline{u}'+{\displaystyle \frac{r}{l}}(1-l\underline{u})(a_{2}\underline{u}-l\underline{u})\\
\\
 & +(1-a_{1})\underline{u}(1-{\displaystyle \frac{1+l}{1-a_{1}}}\underline{u})-(1-a_{1})\underline{u}(1-{\displaystyle \frac{1+l}{1-a_{1}}}\underline{u})]\\
\\
= & l\underline{u}[{\displaystyle \frac{r}{l}}(1-l\underline{u})(a_{2}-l)-(1-a_{1})(1-{\displaystyle \frac{1+l}{1-a_{1}}}\underline{u})]\\
\\
= & l\underline{u}\{{\displaystyle \frac{r}{l}}a_{2}-r-(1-a_{1})+[(1+l)-r(a_{2}-l)]\underline{u}\}\geq0
\end{array}
\]
because of condition (\ref{eq:3.5}).

As for the $w$ component we have 

\[
\begin{array}{ll}
 & \underline{w}''-c\underline{w}'+{\displaystyle \frac{1}{\tau}}(\underline{v}-\underline{w})\\
\\
= & \bar{l}\underline{u}''-c\bar{l}\underline{u}'+{\displaystyle \frac{1}{\tau}}(l\underline{u}-\bar{l}\underline{u})\\
\\
= & \bar{l}\underline{u}[{\displaystyle \frac{1}{\tau\bar{l}}}(l-\bar{l})-(1-a_{1})(1-{\displaystyle \frac{1+l}{1-a_{1}}}\underline{u})]\\
\\
= & \bar{l}\underline{u}[{\displaystyle \frac{1}{\tau\bar{l}}}(l-\bar{l})-(1-a_{1})+(1+l)\bar{u}]>0,
\end{array}
\]
due to the choice of $l$ and $\bar{l}$.

The conclusion of the lemma follows. \end{proof}
\begin{rem}
\label{Remark6}1. For any fixed $c\geq2\sqrt{1-a_{1}}$ and $(\bar{u},\bar{v},\bar{w})^{T}(\xi)$,
$(\underline{u},\underline{v},\underline{w})^{T}(\xi)$, $\xi\in\mathbb{R}$
the respectively upper and lower solutions defined in (\ref{eq:3.2})
and (\ref{eq:3.4}), we have relation $(\bar{u},\bar{v},\bar{w})^{T}(\xi)>(\underline{u},\underline{v},\underline{w})^{T}(\xi)$
for $\xi\in\mathbb{R}.$  In fact $\underline{u}(\xi)=\frac{1-a_{1}}{1+l}\bar{u}(\xi)<\bar{u}(\xi)$,
$\xi\in\mathbb{R}$ where $\underline{u}(\xi)$ and $\bar{u}(\xi)$
are the solutions of (\ref{eq:3.5}) and (\ref{eq:3.1}) respectively,
it then follows $\underline{v}(\xi)=l\underline{u}(\xi)<\bar{u}(\xi)=\bar{v}(\xi)$
and $\underline{w}(\xi)=\bar{l}\underline{u}(\xi)<\bar{u}(\xi)=\bar{w}(\xi)$.

2. The construction of lower solution in Lemma 5 also applies to the
case H1-H2b, where we simply require the condition $0<l<{\displaystyle \frac{ra_{2}}{1-a_{1}+r}}$
to be replaced by $0\leq l<{\displaystyle \min\{\frac{ra_{2}}{1-a_{1}+r}},1\}.$
\end{rem}

\subsection{\label{sub:3.2}Ordered upper and lower solutions under conditions
H1, H2b and H3.}

The following estimates of the solutions of system (\ref{eq:2.6})
is needed in the construction of the upper and lower solutions.
\begin{lem}
\label{lem:LowerH2b}Assume the conditions of Lemma \ref{lem:4}.
Let $(u(\xi),v(\xi))$, $\xi\in\mathbb{R}$, be a solution of (\ref{eq:2.6})
for a fixed $c\geq2\sqrt{1-a_{1}}$, then there exists a constant
$a_{2}^{*}>0$ such that for $a_{2}\geq a_{2}^{*}$, $a_{2}u(\xi)\geq v(\xi)$
for $\xi\in\mathbb{R}$.\end{lem}
\begin{proof}
Noting that $u(\xi)$ and $v(\xi)$ have the exactly same (up to the
first order) exponential decay rate at $-\infty$ , and at $+\infty$
we have $a_{2}u(\xi)>v(\xi)$. The rest of the proof follows easily.
\end{proof}
We now set the upper and lower solution pairs for system (\ref{eq:1.11}). 
\begin{lem}
\label{lem:UpperH2b}Let the parameters satisfy H1 (\ref{Condition H1})
and H2b (\ref{CoditionH2b}), then (\ref{eq:3.6}) consists of a lower
solution for (\ref{eq:1.11}).\end{lem}
\begin{proof}
Similar to the proof of Lemma \ref{lem:6} so we skip it.
\end{proof}
We next set up the upper solution for (\ref{eq:1.11}). 

For each fixed $c\geq2\sqrt{1-a_{1}}$ let $(\hat{u},\hat{v})(\xi)$
be the correspondingly the unique solution of (\ref{eq:3.41}), and
we define the function

\begin{equation}
(\bar{u},\bar{v},\bar{w})(\xi)=(\hat{u},\hat{v},\hat{v})(\xi),\qquad\xi\in\mathbb{R}\label{eq:3.7}
\end{equation}
We then have 
\begin{lem}
\label{lem:a_{2}u}Let the parameters satisfy H1 (\ref{Condition H1}),
H2b (\ref{CoditionH2b}) and H3 (\ref{eq:ConditionH3}), and $a_{2}^{*}$
be given in Lemma \ref{lem:LowerH2b}, then for $a_{2}\geq a_{2}^{*}$
(\ref{eq:3.7}) defines an upper solution of (\ref{eq:1.11}).\end{lem}
\begin{proof}
The proof follows easily from the fact that $(\hat{u},\hat{v})(\xi)$,
$\xi\in\mathbb{R}$ solves (\ref{eq:2.6}) and Lemma \ref{lem:LowerH2b}.
\end{proof}
We next show that such constructed upper and lower solutions are ordered.
The following generalized version of sliding domain method (\cite{Berestycki})
is needed.
\begin{prop}
\label{lem:Sliding11}Let two $C^{2}$ vector functions $\bar{U}(\xi)=(\bar{u}_{1}(\xi),\bar{u}_{2}(\xi),...,\bar{u}_{n}(\xi))$
and $\underline{U}(\xi)=(\underline{u}_{1}(\xi),\underline{u}_{2}(\xi),...,\underline{u}_{n}(\xi))$
satisfy the following inequalities:

\begin{equation}
D\bar{U}''-c\bar{U}'+F(U)\leq0\leq D\underline{U}''-c\underline{U}'+F(\underline{U})\quad\mbox{for\,}\,\xi\in[-N,N]\label{eq:3.8}
\end{equation}
and

\begin{equation}
\underline{U}(-N)<\bar{U}(\xi)\quad\mbox{for\,}\,\xi\in(-N,N],\label{eq:3.9}
\end{equation}

\begin{equation}
\underline{U}(\xi)<\bar{U}(N)\quad\mbox{for\,}\,\xi\in[-N,N),\label{eq:3.10}
\end{equation}
where $D$ is a diagonal matrix with positive entries $D_{i}$, $i=1,2...n$,
$F(U)=(F_{1}(U),...,F_{n}(U))$ is $C^{1}$ with respect to its components
and $\frac{\partial F_{i}}{\partial u_{j}}\geq0$ for $i\neq j$,
$i,j=1,2...n$, then 

\begin{equation}
\underline{U}(\xi)\leq\bar{U}(\xi),\qquad\xi\in[-N,N].\label{eq:3.11}
\end{equation}
\end{prop}
\begin{proof}
We adapt the proof of \cite{Berestycki}. Shift $\bar{U}(\xi)$ to
the left, for $0\leq\mu\leq2N$, consider $\bar{U}^{\mu}(\xi):=\bar{U}(\xi+\mu)$
on the interval $(-N-\mu,N-\mu)$. On both ends of the interval, by
(\ref{eq:3.9}) and (\ref{eq:3.10}), we have

\begin{equation}
\underline{U}(\xi)<\bar{U}^{\mu}(\xi).\label{eq:3.12}
\end{equation}
Starting from $\mu=2N$, decreasing $\mu$, for every $\mu$ in $0<\mu<2N$,
the inequality (\ref{eq:3.12}) is true on the end points of the respective
interval. On decreasing $\mu$, suppose that there is a first $\mu$
with $0<\mu<2N$ such that

\[
\underline{U}(\xi)\leq\bar{U}^{\mu}(\xi)\quad\xi\in(-N-\mu,N-\mu)
\]
and there is one component, for example the $i-th$, such that the
equality holds on a point $\xi_{1}$ inside the interval. Let $W(\xi)=(w_{1}(\xi),w_{2}(\xi),...,w_{n}(\xi))=\bar{U}^{r}(\xi)-\underline{U}(\xi)$,
then $w_{i}(\xi)$, $i=1,2,...,n$ satisfies 

\[
\left\{ \begin{array}{l}
D_{i}w_{i}''-cw_{i}'+\frac{\partial F_{i}}{\partial u_{i}}w_{i}\leq D_{i}w_{i}''-cw_{i}'+\Sigma_{j=1}^{n}\frac{\partial F_{i}}{\partial u_{j}}w_{j}\leq0,\\
\\
w_{i}(\xi_{1})=0,\; w_{j}(\xi)\geq0\:\mbox{\,\ f\mbox{or}}\:\xi\in[-N-\mu,N-\mu],
\end{array}\right.
\]
the Maximum principle further implies that $w_{i}\equiv0$ for $\xi\in[-N-\mu,N-\mu]$,
but this is in contradiction with (\ref{eq:3.12}) on the boundary
points $\xi=-N-\mu$ and $\xi=N-\mu$. So we can decrease $\mu$ all
the way to zero. This proves the conclusion.
\end{proof}
We next shift the upper solution obtained in Lemma \ref{lem:a_{2}u}
far to the left to achieve the orderness between the upper and lower
solutions.
\begin{lem}
Let $c\geq2\sqrt{1-a_{1}}$ be fixed and $(\bar{u},\bar{v},\bar{w})(\xi)$,
$(\underline{u},\underline{v},\underline{w})(\xi)$ are the corresponding
upper and lower solutions derived in Lemma \ref{lem:a_{2}u} and Lemma
\ref{lem:LowerH2b}, then there exists a $\eta_{0}\geq0$ such that
for all $\eta\geq\eta_{0}$ we have 

\[
(\bar{u},\bar{v},\bar{w})(\xi+\eta)\geq(\underline{u},\underline{v},\underline{w})(\xi),\quad\xi\in\mathbb{R}.
\]
\end{lem}
\begin{proof}
On the boundary, we have $(\bar{u},\bar{v},\bar{w})(\xi)\rightarrow(1,1,1)$
and $(\underline{u},\underline{v},\underline{w})(\xi)\rightarrow(\frac{1-a_{1}}{1+l},l\frac{1-a_{1}}{1+l},\bar{l}\frac{1-a_{1}}{1+l})<(1,1,1)$
as $\xi\rightarrow+\infty$ , hence there exists a sufficiently large
$N_{1}>0$ such that for any $\eta\geq0,$ 

\begin{equation}
(\bar{u},\bar{v},\bar{w})(\xi+\eta)>(\underline{u},\underline{v},\underline{w})(\xi),\quad\xi\in[N_{1},\,+\infty).\label{eq:3.13}
\end{equation}

while at $\xi=-\infty$, we have the asymptotics of the upper and
lower solutions:

For $c>2\sqrt{1-a_{1}}$, 

\begin{equation}
\left(\begin{array}{c}
\bar{u}(\xi)\\
\\
\bar{v}(\xi)\\
\\
\bar{w}(\xi)
\end{array}\right)=\left(\begin{array}{c}
A_{1}\\
\\
A_{2}\\
\\
A_{2}
\end{array}\right)e^{\frac{c-\sqrt{c^{2}-4(1-a_{1})}}{2}\xi}+o(e^{\frac{c-\sqrt{c^{2}-4(1-a_{1})}}{2}\xi});\label{eq:3.14}
\end{equation}
and

\begin{equation}
\left(\begin{array}{c}
\underline{u}(\xi)\\
\\
\underline{v}(\xi)\\
\\
\underline{w}(\xi)
\end{array}\right)=\left(\begin{array}{c}
A_{1u}\\
\\
A_{2u}\\
\\
A_{3u}
\end{array}\right)e^{\frac{c-\sqrt{c^{2}-4(1-a_{1})}}{2}\xi}+o(e^{\frac{c-\sqrt{c^{2}-4(1-a_{1})}}{2}\xi});\label{eq:3.15}
\end{equation}

While for $c=2\sqrt{1-a_{1}}$

\begin{equation}
\left(\begin{array}{c}
\bar{u}(\xi)\\
\\
\bar{v}(\xi)\\
\\
\bar{w}(\xi)
\end{array}\right)=\left(\begin{array}{c}
A_{11}+B_{11}\xi\\
\\
A_{21}+B_{21}\xi\\
\\
A_{31}+B_{31}\xi
\end{array}\right)e^{\sqrt{1-a_{1}}\xi}+o(e^{\sqrt{1-a_{1}}\xi})\label{eq:3.16}
\end{equation}
and 
\begin{equation}
\left(\begin{array}{c}
\underline{u}(\xi)\\
\\
\underline{v}(\xi)\\
\\
\underline{v}(\xi)
\end{array}\right)=\left(\begin{array}{c}
A_{11}+B_{11}\xi\\
\\
A_{21}+B_{21}\xi\\
\\
A_{31}+B_{31}\xi
\end{array}\right)e^{\sqrt{1-a_{1}}\xi}+o(e^{\sqrt{1-a_{1}}\xi})\label{eq:3.17}
\end{equation}

Then it is easy to see that there exists a $\eta_{0}\geq0$ such that
for any $\eta\geq\eta_{0}$, we have a $-N_{2}<0$ and the relation

\begin{equation}
(\bar{u},\bar{v},\bar{w})(\xi+\eta)>(\underline{u},\underline{v},\underline{w})(\xi),\quad\xi\in(-\infty,\,-N_{2}]\label{eq:3.18}
\end{equation}
holds for $c\geq2\sqrt{1-a_{1}}$. We may further adjust $N_{1}$
and $N_{2}$ such that (\ref{eq:3.13}) and (\ref{eq:3.18}) hold
on the interval $[N,+\infty)$ and $(-\infty,N]$ respectively for
some large $N>0$. While on the interval $[-N,N]$, $(\bar{u},\bar{v},\bar{w})(\xi+\eta)$
and $(\underline{u},\underline{v},\underline{w})(\xi)$ satisfy the
conditions of Proposition \ref{lem:Sliding11}, we therefore have
\begin{equation}
(\bar{u},\bar{v},\bar{w})(\xi+\eta)\geq(\underline{u},\underline{v},\underline{w})(\xi),\quad\xi\in[-N,N].\label{eq:3.19}
\end{equation}

The conclusion of the lemma follows from (\ref{eq:3.13}), (\ref{eq:3.18})
and (\ref{eq:3.19}). 
\end{proof}
In the sequal we still write the shifted upper solution as $(\bar{u},\bar{v},\bar{w})(\xi)$,
$\xi\in\mathbb{R}$.

\subsection{\label{sub:Monotone-waves-and}Monotone waves and their asymptotics.}

With such constructed ordered upper and lower solution pairs, we now
have
\begin{thm}
\label{thm:2.06}Assume either the conditions H1 and H2a or H1, H2b
and H3, then for every $c\geq2\sqrt{1-a_{1}}$, system (\ref{eq:1.11})
has a unique (up to a translation of the origin) traveling wave solution.
The traveling wave solution is strictly increasing on $\mathbb{R}$
and has the following asymptotic properties:

1. Corresponding to the wave speed $c>2\sqrt{1-a_{1}}$, 
\begin{equation}
\left(\begin{array}{c}
u(\xi)\\
\\
v(\xi)\\
\\
w(\xi)
\end{array}\right)=\left(\begin{array}{c}
A_{1}\\
\\
A_{2}\\
\\
A_{3}
\end{array}\right)e^{\frac{c-\sqrt{c^{2}-4(1-a_{1})}}{2}\xi}+o(e^{\frac{c-\sqrt{c^{2}-4(1-a_{1})}}{2}\xi})\label{eq:3.20}
\end{equation}
 as $\xi\rightarrow-\infty$; 

while Corresponding to the wave speed $c=2\sqrt{1-a_{1}}$, we have

\begin{equation}
\left(\begin{array}{c}
u(\xi)\\
\\
v(\xi)\\
\\
w(\xi)
\end{array}\right)=\left(\begin{array}{c}
(A_{11c}+A_{12c}\xi)\\
\\
(A_{21c}+A_{22c}\xi)\\
\\
(A_{31c}+A_{32c}\xi)
\end{array}\right)e^{\sqrt{1-a_{1}}\xi}+o(\xi e^{\sqrt{1-a_{1}}\xi})\label{eq:3.21}
\end{equation}
 as $\xi\rightarrow-\infty$.

For any speed $c\geq2\sqrt{1-a_{1}}$, we have

\begin{equation}
\left(\begin{array}{c}
u(\xi)\\
\\
v(\xi)\\
\\
w(\xi)
\end{array}\right)=\left(\begin{array}{c}
1\\
\\
1\\
\\
1
\end{array}\right)-\bar{M}(\xi)+o(\bar{M}(\xi))\qquad\xi\rightarrow+\infty,\label{eq:3.22}
\end{equation}
where

\begin{equation}
\begin{array}{lll}
\bar{M}(\xi) & = & s_{1}\left(\begin{array}{c}
1\\
\\
0\\
\\
0
\end{array}\right)e^{\frac{c-\sqrt{c^{2}+4}}{2}\xi}+s_{2}\left(\begin{array}{c}
\frac{a_{1}}{r(1-a_{2})+1}\\
\\
\tau r(1-a_{2})+1\\
\\
1
\end{array}\right)e^{\frac{c-\sqrt{c^{2}+4(a_{2}-1)}}{2}\xi}\\
\\
 &  & +s_{3}\left(\begin{array}{c}
1\\
\\
0\\
\\
1
\end{array}\right)e^{\frac{c-\sqrt{c^{2}+\frac{4}{\tau}}}{2}\xi},
\end{array}\label{eq:3.23}
\end{equation}
and the real numbers $s_{i}$, $i=1,2,3$ are not all zeros and be
such that $-\bar{M}(\xi)+o(\bar{M}(\xi))<0$ for $\xi>0$ large enough.
The speed $c=2\sqrt{1-a_{1}}$ is the minimal wave speed in the sense
that below which there is no monotone waves of (\ref{eq:1.11}).\end{thm}
\begin{proof}
Starting from the upper and lower solution pairs obtained in section
\ref{sub:3.1} and section \ref{sub:3.2} and using the monotone iteration
scheme provided in \cite{WuZou,BoumenirNguyen}, we obtain the existence
of the solution $(u(\xi),v(\xi),w(\xi))$ to (\ref{eq:1.11}) for
every fixed $c\geq2\sqrt{1-a_{1}}$. The solution satisfies 
\[
\left(\begin{array}{c}
\underline{u}(\xi)\\
\\
\underline{v}(\xi)\\
\\
\underline{w}(\xi)
\end{array}\right)\leq\left(\begin{array}{c}
u(\xi)\\
\\
v(\xi)\\
\\
w(\xi)
\end{array}\right)\leq\left(\begin{array}{c}
\bar{u}(\xi)\\
\\
\bar{v}(\xi)\\
\\
\bar{w}(\xi)
\end{array}\right)\qquad\xi\in\mathbb{R}.
\]

Lemma \ref{lem:3} and Lemma \ref{lem:4} imply that the upper- and
the lower-solutions as derived in section \ref{sub:3.1} and section
\ref{sub:3.2} have the same asymptotic rates at $-\infty$. Then
(\ref{eq:3.20}) and (\ref{eq:3.21}) then follow from Lemmas \ref{lem:3}
and \ref{lem:4}.

To derive the asymptotic decay rate of the traveling wave solutions
at $+\infty$, we let $c\geq2\sqrt{1-a_{1}}$ and 
\begin{equation}
U(\xi):=(u(\xi),v(\xi),w(\xi))\quad\xi\in\mathbb{R}\label{eq:3.24}
\end{equation}
 be the corresponding traveling wave solution of (\ref{eq:1.11})
generated from the monotone iteration. We differentiate (\ref{eq:1.11})
with respect to $\xi$, and note that $U'(\xi):=(w_{1},w_{2},w_{3})^{T}(\xi)$
satisfies 
\begin{equation}
(w_{1})_{\xi\xi}-c(w_{1})_{\xi}+A_{11}(u,v,w)w_{1}+A_{12}(u,v,w)w_{2}+A_{13}(u,v,w)w_{3}=0,\label{eq:3.25}
\end{equation}
 
\begin{equation}
(w_{2})_{\xi\xi}-c(w_{2})_{\xi}+A_{21}(u,v,w)w_{1}+A_{22}(u,v,w)w_{2}+A_{23}(u,v,w)w_{3}=0,\label{eq:3.26}
\end{equation}
\\
\begin{equation}
(w_{3})_{\xi\xi}-c(w_{3})_{\xi}+A_{31}(u,v,w)w_{1}+A_{32}(u,v,w)w_{2}+A_{33}(u,v,w)w_{3}=0\label{eq:3.27}
\end{equation}
 where 

\[
\left(\begin{array}{ccc}
A_{11} & A_{12} & A_{13}\\
A_{21} & A_{22} & A_{23}\\
A_{31} & A_{32} & A_{33}
\end{array}\right)=\left(\begin{array}{ccc}
1-a_{1}-2u+a_{1}w & 0 & a_{1}u\\
a_{2}r(1-v) & -r-a_{2}ru+2rv & 0\\
0 & \frac{1}{\tau} & -\frac{1}{\tau}
\end{array}\right).
\]

The limit system of (\ref{eq:3.25}), (\ref{eq:3.26}) and (\ref{eq:3.27})
at $\xi=+\infty$ is 
\begin{equation}
\left\{ \begin{array}{l}
(\psi_{1})_{\xi\xi}-c(\psi_{1})_{\xi}-\psi_{1}+a_{1}\psi_{3}=0,\\
\\
(\psi_{2})_{\xi\xi}-c(\psi_{2})_{\xi}+r(1-a_{2})\psi_{2}=0,\\
\\
(\psi_{3})_{\xi\xi}-c(\psi_{3})_{\xi}+\frac{1}{\tau}\psi_{2}-\frac{1}{\tau}\psi_{3}=0.
\end{array}\right.\label{eq:3.28}
\end{equation}

It is easy to see that system (\ref{eq:3.28}) admits exponential
dichotomy (\cite{Coddington}). Since the traveling wave solution
$(u(\xi),v(\xi),w(\xi))$ converge monotonically to a constant limit
as $\xi\rightarrow\pm\infty$, the derivative of the traveling wave
solution satisfies $(w_{1}(\pm\infty),w_{2}(\pm\infty),w_{3}(\pm\infty))=(0,0,0)$
(\cite{WuZou}, p.$658$ Lemma $3.2$). Hence we are only interested
in finding bounded solutions of (\ref{eq:3.28}) at $+\infty$.

Introducing transformation $\Psi=PY$ by 
\begin{equation}
\left(\begin{array}{c}
\psi_{1}(\xi)\\
\psi_{2}(\xi)\\
\psi_{3}(\xi)
\end{array}\right)=\left(\begin{array}{ccc}
1 & \frac{a_{1}}{r(1-a_{2})+1} & \frac{a_{1}\tau}{\tau-1}\\
0 & \tau r(1-a_{2})+1 & 0\\
0 & 1 & 1
\end{array}\right)\left(\begin{array}{c}
y_{1}(\xi)\\
y_{2}(\xi)\\
y_{3}(\xi)
\end{array}\right)\doteq P\left(\begin{array}{c}
y_{1}(\xi)\\
y_{2}(\xi)\\
y_{3}(\xi)
\end{array}\right)\label{eq:3.29}
\end{equation}
we can decouple (\ref{eq:3.28}) into the following equivalent system: 

\begin{equation}
\left\{ \begin{array}{l}
(y_{1})_{\xi\xi}-c(y_{1})_{\xi}-y_{1}=0,\\
\\
(y_{2})_{\xi\xi}-c(y_{2})_{\xi}+r(1-a_{2})y_{2}=0,\\
\\
(y_{3})_{\xi\xi}-c(y_{3})_{\xi}-\frac{1}{\tau}y_{3}=0,
\end{array}\right.\label{eq:3.30}
\end{equation}
and find its bounded solutions at $+\infty$ explicitely. In fact,
for some nonzero constants $d_{1},d_{2},d_{3}$, we have

\begin{equation}
y_{1}(\xi)=d_{1}e^{\frac{c-\sqrt{c^{2}+4}}{4}\xi},\quad y_{2}(\xi)=d_{2}e^{\frac{c-\sqrt{c^{2}+4r(a_{2}-1)}}{2}\xi},\quad y_{3}(\xi)=d_{3}e^{\frac{c-\sqrt{c^{2}+\frac{4}{\tau}}}{2}\xi}.\label{eq:3.31}
\end{equation}

Transforming back to $\Psi$ we have

\begin{equation}
\left(\begin{array}{c}
\psi_{1}(\xi)\\
\psi_{2}(\xi)\\
\psi_{3}(\xi)
\end{array}\right)=P\left(\begin{array}{c}
y_{1}(\xi)\\
y_{2}(\xi)\\
y_{3}(\xi)
\end{array}\right)=\left(\begin{array}{c}
y_{1}(\xi)+\frac{a_{1}}{r(1-a_{2})+1}y_{2}(\xi)+\frac{a_{1}\tau}{\tau-1}y_{3}(\xi)\\
(\tau r-\tau ra_{2}+1)y_{2}(\xi)\\
y_{2}(\xi)+y_{3}(\xi)
\end{array}\right),\label{eq:3.32}
\end{equation}
Hence we have (\ref{eq:3.22}) on intergrating (\ref{eq:3.32}).

We next show the strict monotonicity of the traveling wave solutions,
which will be a key ingredient in locating the eigenvalues of the
linearized operator about the traveling wave in a separate study.
By the monotone iteration process (see \cite{WuZou}), the traveling
wave solution $U(\xi)$ is increasing for $\xi\in\mathbb{R}$, it
then follows that $(w_{1}(\xi),w_{2}(\xi),w_{3}(\xi))^{T}=U'(\xi)\geq0$
and satisfies (\ref{eq:3.25}), (\ref{eq:3.26}) and (\ref{eq:3.27}).
The monotonicity of system (\ref{eq:1.11}) and the Maximum Principle
imply that $(w_{1},w_{2},w_{3})^{T}(\xi)>0$ for $\xi\in\mathbb{R}$.
This concludes the strict monotonicity of the traveling wave solutions.

On the uniqueness of the traveling wave solution for every $c\geq2\sqrt{1-a_{1}}$,
we only prove the conclusion for traveling wave solutions with asymptotic
rates given in (\ref{eq:3.20}) and (\ref{eq:3.22}), since other
case can be proved similarly. Let $U_{1}(\xi)=(u_{1},v_{1},w_{1})^{T}$
and $U_{2}(\xi)=(u_{2},v_{2},w_{2})^{T}$ be two traveling wave solutions
of system (\ref{eq:1.11}) with the same speed $c>2\sqrt{1-a_{1}}$.
There exist positive constants $A_{i}$, $B_{i}$, $i=1,2,3,4$ and
a large number $N>0$ such that for $\xi<-N$,
\begin{equation}
U_{1}(\xi)=\left(\begin{array}{c}
A_{1}\\
\\
A_{2}\\
\\
A_{3}
\end{array}\right)e^{\frac{c+\sqrt{c^{2}-4(1-a_{1})}}{2}\xi}+o(e^{\frac{c+\sqrt{c^{2}-4(1-a_{1})}}{2}\xi})\label{eq:3.33}
\end{equation}
 
\begin{equation}
U_{2}(\xi)=\left(\begin{array}{c}
A_{4}\\
\\
A_{5}\\
\\
A_{6}
\end{array}\right)e^{\frac{c+\sqrt{c^{2}-4(1-a_{1})}}{2}\xi}+o(e^{\frac{c+\sqrt{c^{2}-4(1-a_{1})}}{2}\xi});\label{eq:3.34}
\end{equation}
 and for $\xi>N$,
\begin{equation}
U_{1}(\xi)=\left(\begin{array}{c}
{\displaystyle 1-\bar{B}_{1}e^{\underline{\mu}_{1}\xi}}\\
\\
{\displaystyle 1-\bar{B}_{2}e^{\underline{\mu}_{2}\xi}}\\
\\
{\displaystyle 1-\bar{B}_{3}e^{\underline{\mu}_{3}\xi}}
\end{array}\right)+\left(\begin{array}{c}
{\displaystyle o(e^{\underline{\mu}_{1}\xi})}\\
\\
{\displaystyle o(e^{\underline{\mu}_{2}\xi}})\\
\\
{\displaystyle o(e^{\underline{\mu}_{3}\xi}})
\end{array}\right),\label{eq:3.35}
\end{equation}
 
\begin{equation}
U_{2}(\xi)=\left(\begin{array}{c}
{\displaystyle 1-\underline{B}_{1}e^{\underline{\mu}_{1}\xi}}\\
\\
{\displaystyle 1-\underline{B}_{2}e^{\underline{\mu}_{2}\xi}}\\
\\
{\displaystyle 1-\underline{B}_{3}e^{\underline{\mu}_{3}\xi}}
\end{array}\right)+\left(\begin{array}{c}
{\displaystyle o(e^{\underline{\mu}_{1}\xi})}\\
\\
{\displaystyle o(e^{\underline{\mu}_{2}\xi}})\\
\\
{\displaystyle o(e^{\underline{\mu}_{3}\xi}})
\end{array}\right),\label{eq:3.36}
\end{equation}
where $\underline{\mu}_{i}$ is one of the elements in the set \{$\frac{c-\sqrt{c^{2}+4(1-a_{1})}}{2},$
$\frac{c-\sqrt{c^{2}+4}}{2}$ , $\frac{c-\sqrt{c^{2}+4/\tau}}{2}$\},
and $\bar{B}_{i}$, $\underline{B}_{i}$ are positive numbers, $i=1,2,3$.
The traveling wave solutions of system (\ref{eq:1.11}) are translation
invariant, thus for any $\theta>0$, $U_{1}^{\theta}(\xi):=U_{1}(\xi+\theta)$
is also a traveling wave solution of (\ref{eq:1.11}). By (\ref{eq:3.33})
and (\ref{eq:3.35}), the solution $U_{1}(\xi+\theta)$ has the asymptotics
\begin{equation}
U_{1}^{\theta}(\xi)=\left(\begin{array}{c}
A_{1}e^{\frac{c+\sqrt{c^{2}-4(1-a_{1})}}{2}\theta}e^{\frac{c+\sqrt{c^{2}-4(1-a_{1})}}{2}\xi}\\
\\
A_{2}e^{\frac{c+\sqrt{c^{2}-4(1-a_{1})}}{2}\theta}e^{\frac{c+\sqrt{c^{2}-4(1-a_{1})}}{2}\xi}\\
\\
A_{3}e^{\frac{c+\sqrt{c^{2}-4(1-a_{1})}}{2}\theta}e^{\frac{c+\sqrt{c^{2}-4(1-a_{1})}}{2}\xi}
\end{array}\right)+o(e^{\frac{c+\sqrt{c^{2}-4\alpha}}{2}\xi})\label{eq:3.37}
\end{equation}
 for $\xi\leq-N$;
\begin{equation}
U_{1}^{\theta}(\xi)=\left(\begin{array}{c}
{\displaystyle 1-\bar{B}_{1}e^{\underline{\mu}_{1}\theta}e^{\underline{\mu}_{1}\xi}}\\
\\
{\displaystyle 1-\bar{B}_{2}e^{\underline{\mu}_{2}\theta}e^{\underline{\mu}_{2}\xi}}\\
\\
{\displaystyle 1-\bar{B}_{3}e^{\underline{\mu}_{3}\theta}e^{\underline{\mu}_{3}\xi}}
\end{array}\right)+\left(\begin{array}{c}
{\displaystyle o(e^{\underline{\mu}_{1}\xi})}\\
\\
{\displaystyle o(e^{\underline{\mu}_{2}\xi}})\\
\\
{\displaystyle o(e^{\underline{\mu}_{3}\xi}})
\end{array}\right)\label{eq:3.38}
\end{equation}
 for $\xi\geq N$.

Choosing $\theta>0$ large enough such that 
\begin{equation}
A_{i}e^{\frac{c+\sqrt{c^{2}-(1-a_{1})}}{2}\theta}>A_{3+i},\qquad i=1,2,3,\label{eq:3.39}
\end{equation}
 
\begin{equation}
\bar{B}_{i}e^{\underline{\mu}_{i}\theta}<\underline{B}_{i},\qquad i=1,2,3\label{eq:3.40}
\end{equation}
 then one has 
\begin{equation}
U_{1}^{\theta}(\xi)>U_{2}(\xi)\label{eq:3.41}
\end{equation}
 for $\xi\in(-\infty,-N]$$\cup$$[N+\infty).$ We now consider system
(\ref{eq:1.3}) on $[-N,+N]$. We can verify that $U_{1}^{\theta}(\xi)\doteq\bar{U}(\xi)$
and $U_{2}(\xi)\doteq\underline{U}(\xi)$ satisfy all the conditions
of Proposition \ref{lem:Sliding11}, hence we have $U_{1}^{\theta}(\xi)\geq U_{2}(\xi)$
for $\xi\in[-N,N]$. Further applying the Maximum Principle and noting
that $U_{1}^{\theta}(\pm N)>U_{2}(\pm N)$, we have $U_{1}^{\theta}(\xi)>U_{2}(\xi)$
for $\xi\in[-N,N]$ .

Consequently we have $U_{1}^{\theta}(\xi)>U_{2}(\xi)$ on $\mathbb{R}$.

Now, decrease $\theta$ until one of the following situations happens.

1. There exists a $\bar{\theta}\geq0$, such that $U_{1}^{\bar{\theta}}(\xi)\equiv U_{2}(\xi)$.
In this case we have finished the proof.

2. There exists a $\bar{\theta}\geq0$ and $\xi_{1}\in\mathbb{R}$,
such that one of the components of $U^{\bar{\theta}}$ and $U_{2}$
are equal there; and for all $\xi\in\mathbb{R}$, we have $U_{1}^{\bar{\theta}}(\xi)\geq U_{2}(\xi)$.
On applying the Maximum Principle on $\mathbb{R}$ for that component,
we find $U_{1}^{\bar{\theta}}$ and $U_{2}$ must be identical on
that component. To fix ideas, we suppose that the component is the
first component. Then $U_{1}^{\bar{\theta}}-U_{2}$ satisfies (\ref{eq:3.25}),
(\ref{eq:3.26}) and (\ref{eq:3.27}). Plugging $w_{1}\equiv0$ into
(\ref{eq:3.26}) we find that there is at least one $\xi_{\bar{\theta}}$
such that $w_{2}(\xi_{\bar{\theta}})=0$. Then by applying maximum
principle to (\ref{eq:3.26}), we have $w_{2}(\xi)\equiv0$ for $\xi\in\mathbb{R}$.
Similarly we also find $w_{3}(\xi)\equiv0,\,\xi\in\mathbb{R}$. We
have then returned to case 1.

Hence, in either situation, there exists a $\bar{\theta}\geq0$, such
that 
\[
U_{1}^{\bar{\theta}}(\xi)\equiv U_{2}(\xi).
\]
 for all $\xi\in\mathbb{R}$. 

The nonexistence of the monotone traveling waves for (\ref{eq:1.11})
comes from the fact that all its solutions are oscillatory for $c\leq2\sqrt{1-a_{1}}$.
\end{proof}
Concerning the wave solutions of system (\ref{eq:1.2}) we immediately
have
\begin{cor}
Assume the conditions H1 (\ref{Condition H1}) and H2a (\ref{ConditionH2a})
or conditions H1 (\ref{Condition H1}), H2b (\ref{CoditionH2b}) and
H3 (\ref{eq:ConditionH3}). Then for each $c\geq2\sqrt{1-a_{1}}$
and all $\tau>0$ system (\ref{eq:1.2}) has a unique monotonic traveling
wave solution connecting the equilibrium $(0,1)$ with $(1,0).$
\end{cor}
The conclusion of the Corollary says the delay does not change the
course of the traveling waves, but it may change the asymptotic behaviors
of the wave solutions at $+\infty$.

\end{document}